\documentclass[12pt,twoside]{amsart}
\usepackage[dvips]{graphics}
\usepackage[all]{xy}
\usepackage[T1]{fontenc}
\usepackage[sc]{mathpazo}
\usepackage{times}
\usepackage{a4,amsmath,amssymb,amsfonts,amscd,mathrsfs}
\usepackage{euscript}
\reversemarginpar
\newtheorem{lemma}[subsection]{Lemma}
\newtheorem{proposition}[subsection]{Proposition}

\newtheorem*{conjecture*}{Conjecture}
\newtheorem*{example*}{Example}
\newtheorem{remark}{Remark}

\newcommand{\Mod}{{\rm Mod}}
\newcommand{\Orth}{{\rm O}}
\newcommand{\GL}{{\rm GL}}
\newcommand{\SL}{{\rm SL}}

\renewcommand{\Re}{\operatorname{Re}}

\newcommand{\dee}{\partial}

\newcommand{\inv}{^{-1}}
\newcommand{\matzwei}[1]{\left(\begin{array}{rr}#1\end{array} \right)}

\newcommand{\ol}[1]{\overline{#1}}
\renewcommand{\phi}{\varphi}
\newcommand{\Quad}{{\rm Quad}}

\newcommand{\sca}[1]{\left\langle #1 \right\rangle}

\newcommand{\tr}{{\rm tr}}
\newcommand{\transp}{\mbox{}^t}
\newcommand{\ul}[1]{\underline{#1}}

\newcommand{\lvect}[1]{\left(\begin{array}{l}#1\end{array} \right)}

\DeclareMathOperator{\Diff}{D}
\newcommand{\cdop}{{\mathbb C}}
\newcommand{\edop}{{\mathbb E}}

\newcommand{\hdop}{{\mathbb H}}

\newcommand{\qdop}{{\mathbb Q}}
\newcommand{\rdop}{{\mathbb R}}
\newcommand{\zdop}{{\mathbb Z}}

\newcommand{\Ocal}{{\mathcal O}}
\author{Juan Marcos Cervi\~no}
\address{Mathematisches Institut, Universit\"at Heidelberg, 69120
Heidelberg, Germany}
\email{juan.cervino@iwr.uni-heidelberg.de}
\author{Georg Hein}
\address{Fakult\"at Mathematik, Universit\"at Duisburg-Essen, 45117 Essen,
Germany}
\email{georg.hein@uni-due.de}
\date{June 17, 2011}
\begin{document}
\title{quadratic forms and their theta series -- infinitesimal
aspects}
\begin{abstract}
We study the theta map which assigns to a real quadratic form its theta
series.  We introduce two invariants reflecting whether the differential of
the theta map vanishes or is degenerate. We provide examples of lattices where
this differential is zero.  These invariants turn out to be modular forms for
integral lattices.  We illustrate this in the rank two case.
\end{abstract}
\maketitle

\section{Preliminaries and notation}
Let $\Quad_n(\rdop)$ be the real vector space of quadratic forms in $n$
variables. It is a
$\frac{n^2+n}{2}$-dimensional vector space which we identify with the
space of symmetric $n \times n$
matrices with real entries. Inside this vector space is the open cone
$\Quad_n^+(\rdop)$ of positive
definite forms.

The second space of interest is the ring $\Mod$ of holomorphic functions
on the upper half plane $\hdop$
which satisfy a certain growth condition. That is
\[ \Mod= \left\{ f : \hdop \to \cdop \, \left|
\begin{array}{l}
\text{there exists a discrete subset } M \subset \rdop_{\geq 0}\\
\text{and a function } a:M \to \cdop \text{ with } m \mapsto a_m \\
\text{bounded by two polynomials } H_1 \text{ and } H_2.\\
\text{This means, that for all }  d \in \rdop \text{ we have}\\
\#\{ m\in M | m \leq d \} \leq H_1(d) \text{ , and}\\
 \|a_m\| \leq H_2(m) \text{ for all } m \in M. \\
\text{The holomorphic function } f \text{ has an expansion}\\
f(z) = \sum_{m \in M} a_m\exp(2\pi i mz) \, .
\end{array} \right. \right\} .\]
Modular forms are typical elements of $\Mod$.  Moreover, Mod is closed
under differentiation.  We are
interested in the map $\Theta$ which assigns to a quadratic form $A \in
\Quad_n^+(\rdop)$ its theta series.
With the above notation $\Theta$ is a map
\[ \Theta: \Quad_n^+(\rdop) \to  \Mod \quad A \mapsto \Theta_A 
\quad \text{with } \Theta_A(z) = \sum_{\lambda \in \zdop^n} \exp(2 \pi i
(\transp\lambda A
\lambda)z) .\]
Two quadratic forms $A$ and $A'$ are equivalent, if and only if there
exists a $T \in \GL_n(\zdop)$ such
that $A' = \transp T A T$.  We call $A$ and $A'$ isospectral when their
theta series coincide.  Obviously,
we have $\Theta_A=\Theta_{A'}$ for two equivalent quadratic forms, in
words: equivalence implies
isospectrality.  The global Torelli theorem for the $\Theta$-map asks,
whether we can conclude from
$\Theta_A=\Theta_{A'}$ the equivalence of the quadratic forms. A local
Torelli statement investigates
whether $\Theta$ reflects all infinitesimal deformations of a given
quadratic form.  A.~Schiemann showed in
\cite{Sch2} that we have a 

\vspace*{.5em}
{\bf Global Torelli theorem for the $\Theta$ map in small dimensions.}
{\em If the rank of two quadratic forms is at most three, then they are
isospectral if and only if they are
equivalent.}
 
\vspace*{.5em}
Moreover, Schiemann gave an example for the failure of a global Torelli
for quadratic forms of rank four by
providing two inequivalent quadratic forms with the same theta series. 
Schiemann's example from
\cite{Sch1} was generalized to a family of pairs of isospectral quadratic
forms by Conway and Sloane in
\cite{CS}. Conway and Sloane conjectured that pairs of lattices in this
family were pairs of
inequivalent lattices.  The authors showed in \cite{CH3} that the
invariant $\Theta_{1,1}$ developed in
\cite{CH1} allows to distinguish the two lattices for all non-isometric
pairs in the family. Therefore we believe, that
if we add to the $\Theta$ map some of the invariants developed in
\cite{CH1} and \cite{CH2}, then a global
Torelli theorem holds true. We have for example:

\vspace*{.5em}
{\bf Conjecture.}
{\em Two quadratic forms $A$ and $A'$ of rank four are equivalent, if and
only if $\Theta_A=\Theta_{A'}$
and $\Theta_{1,1;A}=\Theta_{1,1;A'}$ holds.}

\vspace*{.5em}
In the following $\Quad_n$ denotes $\Quad_n(\rdop)$ (equivalently for
$\Quad_n^+$).
The aim of this paper is to study the infinitesimal behavior of the map
$\Theta$, in other words the
differential $\Diff\Theta$. For a lattice $A \in \Quad_n^+$ we restrict
to the hyperplane $T_A^0$ of deformations in tangent space
which leave the discriminant of $A$ unchanged. In this paper we answer the
following questions:
\begin{enumerate}
\item Are there quadratic forms $A$ such that the differential of the
$\Theta$ map is zero on the
hyperplane $T^0_A$ in the tangent space?
\item For which quadratic forms $A$ is the differential of the $\Theta$
map at $A$ degenerate?
\item\label{it:3} Is the differential of the $\Theta$ map injective for a
general quadratic form $A$?
\end{enumerate}

It turns out that $\Diff\Theta\vert_{T_A^0}$ is zero, if and only if our
invariant $\Theta_{1,1;A}$ vanishes (see
Proposition \ref{max-deg}(iii)).  Furthermore, we give examples of
quadratic forms with $\Theta_{1,1;A}=0$
coming from $p$-th roots of unity (Proposition \ref{propLp}), and from
large automorphism groups
(Proposition \ref{big-auto}) as root lattices for example.
This implies that there cannot be a local
Torelli theorem. Thus, the question
\eqref{it:3} has a negative answer when we do not restrict to a general
quadratic form $A$. We introduce a new
lattice invariant $\det^2\Diff \Theta$ which is also also a modular form
for integral quadratic forms and
vanishes exactly when the differential $\Diff\Theta$ is degenerate (cf.
Proposition \ref{HLambda}).  A simple
argument gives the local Torelli theorem (Proposition \ref{GenLocTor}) for
a general lattice.  In the
last section we compute these invariants for binary quadratic forms, and
identify all the forms with
vanishing or degenerate differential.

\section{The metric structure on the tangent space of $\Quad_n^+$}
\subsection{Quadratic forms and lattices}
For a point $A \in \Quad_n^+$ we have a positive definite pairing on the
tangent space
$T_A=T_{\Quad_n^+,A}=\Quad_n$ given by the Killing form \[ \sca{ H_1,H_2}
= 2 \cdot  \tr (A\inv H_1 A\inv
H_2) \, .\] Considering $\rdop^n$ with the scalar product given by $A$
there exists a linear transformation
$Q: \rdop^n \to \edop^n$ from $(\rdop^n,A)$ to the standard Euclidean
space which is an isometry. On the
level of matrices we have $A = \transp Q \cdot Q$. Sometimes we consider
instead of $A$ the lattice
$\Lambda=\Lambda_A$ generated by the column vectors of $Q$. The matrix $Q$
is given only up to an element
of the orthogonal group $\Orth(n)$. However, when we require $Q$ to be
upper triangular with positive
diagonal entries, then $Q$ is uniquely given (Cholesky decomposition).

For two quadratic forms $h_1$ and $h_2$ on $\edop^n$ we have the positive
definite pairing
$\sca{h_1,h_2}=h_1(\frac{\dee}{\dee x_1},\frac{\dee}{\dee x_2},
\ldots,\frac{\dee}{\dee x_n})h_2(x_1,x_2,
\ldots ,x_n)$. If we represent $h_1$ and $h_2$ by symmetric $n \times n$
matrices, then we see that the
pairing is given by $ \sca{h_1,h_2}=2 \cdot  \tr(h_1 \cdot h_2)$. Since
the quadratic form $h_i$ on
$\rdop^n$ corresponds to the quadratic form $\transp Q\inv \cdot h_i\cdot
Q\inv$ and the trace is invariant
under conjugation, we obtain the above positive definite form. Under this
identification the quadratic form
$A$ on $\rdop^n$ corresponds to the square length on $\edop^n$.\\
Thus, $H$ defines a harmonic quadratic form on $\rdop^n$ with metric given
by $A$ if and only if $\tr(A\inv
H)=0$.

\subsection{The invariant $\Theta_{1,1;A}$}
In \cite{CH1} the authors defined the lattice invariant
$\Theta_{1,1;\Lambda}$ for a lattice $\Lambda
\subset \edop^n$.  Since it is invariant under the orthogonal group
$\Orth(n)$ it gives an invariant of the
associated quadratic form. For the convenience of the reader we recall its
definition (see Theorem 4.2 in
\cite{CH1}).

Let $A: \zdop^n \to \rdop$ be a positive definite quadratic form,
corresponding to the lattice $\Lambda
\subset \edop^n$. The holomorphic function $\Theta_{1,1;A} =
\Theta_{1,1;\Lambda}$ on
the upper half plane is given by

\[\Theta_{1,1;\Lambda}(\tau) = \sum_{m \geq 0}a_m \exp(2 \pi i m \tau)
\quad \text{with} \quad
a_m= \hspace{-14pt} \sum_{\tiny \begin{array}{c} (\gamma,\delta) \in
\Lambda \times \Lambda\\
\|\gamma\|^2+\|\delta \|^2=m\\
\end{array}} \hspace{-14pt} \left(\frac{\cos^2(\measuredangle
(\gamma,\delta ))}{2}-\frac{1}{2n}\right)\|\gamma\|^2\|\delta \|^2. \]
The function $\Theta_{1,1;\Lambda}$ can also be computed using theta
series with harmonic
coefficients. For a lattice $\Lambda \subset \edop^n$ and a harmonic
function $h:
\edop^n \to \cdop$   we define
\[ \Theta_{h;\Lambda}(\tau):=\sum_{\lambda \in \Lambda} h(\lambda)
\exp(2\pi i \|\lambda\|^2 \tau) \, .\]
We can compute $\Theta_{1,1;\Lambda}$ in terms of the functions
$\Theta_{h;\Lambda}$ as follows: Let $\{
h_i\} _{i=1,\ldots ,N}$ be an orthonormal basis of the harmonic quadratic
forms on $\edop^n$. Then we have
an equality (see \cite[Theorem 3.3]{CH2})
\[ \Theta_{1,1;\Lambda} = \sum_{i=1}^N \Theta_{h_i,\Lambda}^2 \, .\]

\subsection{The splitting of the tangent space $T_A$}\label{split}
We split the tangent space as follows:
\[T_A = T_{\Quad^+_n,A}= \Quad_n = \rdop \cdot A \oplus
T_A^0  \text{ where } T_A^0 = \{ A' \in \Quad_n \, \mid \, \tr(A' \cdot
A\inv) =0
\}.\]
The hyperplane $T_A^0$ describes the infinitesimal deformations with fixed
discriminant.\\
We compute
\[ \dee_B\Theta_A:=\frac{1}{2 \pi i z} \frac{\dee}{\dee t} 
\Theta_{A+tB}\vert_{t=0} = \sum_{\lambda \in
\zdop^n} \transp\lambda B \lambda \exp(2 \pi i(\transp\lambda A \lambda)z)
\, .\]
Let $A$ be an integral quadratic form, which means $\Theta_A \in
\zdop[[q]]$. It follows that $\Theta_A$
is a modular form for a subgroup $\Gamma \subset \SL_2(\zdop)$.  For a
harmonic quadratic form $B$ we have
that $\dee_B\Theta_A$ is a modular form of weight $\frac{n+4}{2}$,
whereas $\dee_A\Theta_A = \Theta_A'$.
We conclude the following
\begin{proposition}\label{max-deg}
Let $A$ be an integer valued quadratic form of rank $n$ and level $l$.
\begin{enumerate}
\item $\Theta_A$ is a modular form of weight $\frac{n}{2}$ and level
$l$.\\
\item $\dee_B \Theta_A$ is a modular form  of weight $\frac{n+4}{2} \iff
B \in
T_A^0$. In particular, we have $\dee_B \Theta_A=0$ only for $B\in
T_A^0$.\\
\item The differential $\Diff\Theta$ vanishes identically on $T_A^0 \iff
\Theta_{1,1;A} =0$.
\end{enumerate}
\end{proposition}

\begin{proof}
Only (iii) is new.
Let $B_1,B_2,\ldots B_N$ be a orthonormal basis of $T_A^0$. The functions
$\dee_{B_i}\Theta_A$ take real
values on the imaginary axis in the upper half plane. Thus, since harmonic
quadratic forms are precisely
the elements of $T_A^0$ (cf. Section 2.1), we have an equivalence 
\begin{eqnarray*}
\Diff\Theta\vert_{T_A^0} \equiv 0 
& \iff & \dee_{B_i}\Theta_A = 0 \text{ for all } i=1,\ldots ,N\\
& \iff
& \left( \dee_{B_i}\Theta_A \right) ^2 = 0 \text{ for all } i=1,\ldots ,N\\
& \iff & \sum_{i=1}^N \left( \dee_{B_i}\Theta_A \right)^2 = 0 \, .\\
\end{eqnarray*}
Since $\Theta_{1,1;A}=  \sum_{i=1}^N \left( \dee_{B_i}\Theta_A\right)^2$,
this shows
the assertion.
\end{proof}

\section{Lattices with vanishing differential}\label{VanDiff}
\subsection{First examples for lattices with vanishing
differential}\label{examp}
Let us start with three quadratic forms with vanishing differential. The
vanishing of the modular form
$\Theta_{1,1;\Lambda}$ follows in all three cases from Proposition
\ref{big-auto}. Moreover, for the second
example it also follows from Proposition \ref{propLp} when setting $p=3$.

{\bf Example 1: The Gaussian integers ($A_1^2$).}
The Gaussian integers $\zdop[i] \subset \cdop = \edop^2$ form a lattice. 
The lattice corresponds to the
quadratic form $A_1^2=\matzwei{1&0\\0&1}$.  Its theta series is given by 
$\Theta_{A_1^2}(z) = 1 +4q^{1} +4q^{2} +4q^{4} +8q^{5} +4q^{8} +4q^{9}
+8q^{10} +8q^{13} +4q^{16} +8q^{17}
+ \ldots$ with $q=\exp(2\pi i z)$. We have that $\Theta_{1,1;A_1^2}= 0$.

{\bf Example 2: The Eisenstein integers ($A_2$).}
The Eisenstein integers $\zdop[ \frac{1+\sqrt{-3}}{2}]$ again form a
lattice in $\edop^2$ corresponding to
the quadratic form $A_2=\frac{1}{2}\matzwei{2&1\\1&2}$.  We compute its
theta series $\Theta_{A_2}(z) = 1 +6q^{1}
+6q^{3} +6q^{4} +12q^{7} +6q^{9} +6q^{12} +12q^{13} +6q^{16} +12q^{19}  +
\ldots$ and $\Theta_{1,1;A_2}=
0$.

{\bf Example 3: The $E_8$ lattice.} Here we have
$\Theta_{E_8}(z)=1+240q^2+2160q^4+6720q^6+17520q^8+30240q^{10}+\ldots$.
For the unimodular lattice $E_8$ it was shown
in \cite[Example 3.4]{CH2} that $\Theta_{1,1;E_8} =0$.

\subsection{Lattices with vanishing differential from $p$-th roots of
unity}
Let $p$ be an odd prime, and $\zeta = \exp(2 \pi i/p) \in \cdop$ be a
$p$-th root of unity, and
$K=\qdop(\zeta)$.  The ring $\Ocal_K=\zdop[\zeta]$ possesses $p-1$
embeddings into $\cdop$.  They form
$\frac{p-1}{2}$ conjugated pairs. Choosing one representative from each
pair we obtain the Minkowski
embedding
\[ \iota: \Ocal_K \to \cdop^\frac{p-1}{2} \quad
\text{ given by } \zeta^k \mapsto
\lvect{\zeta^k\\\zeta^{2k}\\\vdots\\\zeta^{\frac{p-1}{2}k}} \, .\]
The embedding of $\Ocal_K \to \cdop$ which sends $\zeta$ to $\zeta^k$ is
denoted by $\sigma_k$.  Let us
identify $\cdop^\frac{p-1}{2}$ with the euclidean space $\edop^{p-1}$, and
denote by $\Lambda_p$ the image
of $\Ocal_K$ in $\edop^{p-1}$. The lattice $\Lambda_p$ corresponds to the
quadratic form given by
\begin{eqnarray}\label{LpAp}
A_p = \frac{1}{2}\left(
\begin{array}{cccccc}
p-1 & -1 & -1 & \cdots & -1 & -1\\
-1 & p-1 & -1 & \cdots & -1 & -1\\
-1 & -1 & p-1 & \cdots & -1 & -1\\
\vdots & \vdots & & \ddots && \vdots\\
-1 & -1 & -1 & \cdots & p-1 & -1\\
-1 & -1 & -1 & \cdots & -1 & p-1\\
\end{array}\right) .
\end{eqnarray}
On $\cdop^\frac{p-1}{2}$ operates the cyclic group $G$ with $p$ elements
and generator $g$ acting as
multiplication by $\zeta^l$ on the $l$-th components of
$\cdop^\frac{p-1}{2}$. Since all eigenvalues of
$g$ are different from one we obtain a free action on $\Lambda_p \setminus
0$.  Having in mind that
$\Lambda_P$ is the image of $\Ocal_K$ under the Minkowski embedding we
obtain an action of the Galois group
$G_{K/\qdop}$ on $\Lambda_P$. Since $G_{K/\qdop}$ is abelian we have
$\sigma(\iota(x))=\iota(\sigma(x))$.
Using these actions we will show that $\Theta_{1,1;\Lambda_p}=0$.  In
order to prepare it, we compute for
$x,y \in \Ocal_K$  the following sum:
\begin{eqnarray}\label{sxy}
s(x,y)&: = &\sum_{k=0}^{p-1} \sum_{\sigma \in G_{K/\qdop}}
\sca{ g^k(\iota(x)),\sigma(\iota(y))}^2 \, .
\end{eqnarray}
The scalar product for $z=(z_i)_{i=1,\ldots, \frac{p-1}{2}}$ and
$z'=(z'_i)_{i=1,\ldots,
\frac{p-1}{2}}$ in $\cdop^{\frac{p-1}{2}}$ induced by its identification
with $\edop^{p-1}$ is given by
$\sca{z,z'}=\frac{1}{2} \sum_{i=1}^{\frac{p-1}{2}} \left( z_i
\ol{z'_i}+\ol{z_i}z'_i \right)$. We obtain
\begin{eqnarray*}
s(x,y)& =& \sum_{k=0}^{p-1} \sum_{\sigma \in G_{K/\qdop}} \left(
\frac{1}{2}\sum_{i=1}^{\frac{p-1}{2}}
\left(\zeta^{ki} \sigma_i(x) \ol{\sigma_i(\sigma(y))}+
\ol{\zeta^{ki} \sigma_i(x)} \sigma_i(\sigma(y)) \right)
\right)^2\\
& =& \frac{1}{4} \sum_{k=0}^{p-1} \sum_{\sigma \in G_{K/\qdop}}
\sum_{i,j=1}^{\frac{p-1}{2}}
\zeta^{k(i+j)} \sigma_i(x)\sigma_j(x)
\ol{\sigma_i(\sigma(y))\sigma_j(\sigma(y))}+ \\
& & \frac{1}{4} \sum_{k=0}^{p-1} \sum_{\sigma \in G_{K/\qdop}}
\sum_{i,j=1}^{\frac{p-1}{2}}
\zeta^{k(i-j)} \sigma_i(x) \ol{\sigma_j(x)
\sigma_i(\sigma(y))}\sigma_j(\sigma(y))+ \\
& & \frac{1}{4} \sum_{k=0}^{p-1} \sum_{\sigma \in G_{K/\qdop}}
\sum_{i,j=1}^{\frac{p-1}{2}}
\zeta^{k(-i+j)} \ol{\sigma_i(x)}\sigma_j(x)
\sigma_i(\sigma(y))\ol{\sigma_j(\sigma(y))}+ \\
& & \frac{1}{4} \sum_{k=0}^{p-1} \sum_{\sigma \in G_{K/\qdop}}
\sum_{i,j=1}^{\frac{p-1}{2}}
\zeta^{k(-i-j)} \ol{\sigma_i(x)\sigma_j(x)}
\sigma_i(\sigma(y))\sigma_j(\sigma(y)) \, . \\
\end{eqnarray*}
Since for any integer $m$ we have $\sum_{k=0}^{p-1} \zeta^{km} = \left\{
\begin{array}{ll} p & \text{ for }
p\mid m\\ 0 & \text{ otherwise }\\ \end{array}\right.$, we conclude that
the first and last summands are zero.
For the same reason we see that in the second and third summand only for
$j=i$ we have a non trivial
contribution. This yields
\[s(x,y)  = \frac{p}{2} \sum_{\sigma \in G_{K/\qdop}}
\sum_{i=1}^{\frac{p-1}{2}}  \|\sigma_i(x)\|^2 \|\sigma_i(\sigma(y))\|^2
\,.  \]
When $\sigma$ runs through $G_{K/\qdop}$ the values
$\|\sigma_i(\sigma(y))\|$ run through the set of values
$\{\|\sigma_j(y) \|\}_{j=1,\ldots,\frac{p-1}{2}}$. Eventually we obtain:
\begin{eqnarray}\label{fsxy}
s(x,y)  = p \sum_{i=1}^{\frac{p-1}{2}} \sum_{j=1}^{\frac{p-1}{2}}
\|\sigma_i(x)\|^2 \|\sigma_j(y)\|^2 = p \|\iota(x)\|^2 \|\iota(y)\|^2 \,
.
\end{eqnarray}
\begin{proposition}\label{propLp}
For the lattice $\Lambda_p$ corresponding to the quadratic form $A_p$
given in (\ref{LpAp}) the modular
form $\Theta_{1,1;\Lambda_p}$ equals zero.
\end{proposition}
\begin{proof}
We write
\[ \Theta_{1,1;\Lambda_p}(z)   =  \sum_{x,y \in \Ocal_K} \left(
\sca{\iota(x),\iota(y)}^2 - \frac{1}{p-1}\|\iota(x)\|^2
\|\iota(y)\|^2\right)
q^{\|\iota(x)\|^2 + \|\iota(y)\|^2} \, . \]
If we consider for a fixed pair $(x,y) \in \Ocal_K$ the sum 
\[ s:=\sum_{g \in G} \sum_{\sigma \in G_{K/\qdop}}
\sca{g(\iota(x)),\iota(\sigma(y))}^2 - \frac{1}{p-1}\|g(\iota(x))\|^2
\|\iota(\sigma(y))\|^2 \, ,\]
then, since both groups act as isometries, we obtain using the definition
of $s(x,y)$ in (\ref{sxy}):
\[ s=s(x,y)-p\|\iota(x)\|^2 \|\iota(y)\|^2 \, .\]
Our formula (\ref{fsxy}) for $s(x,y)$ shows that $s=0$.
\end{proof}

\subsection{Large isometry groups imply the vanishing of
$\Theta_{1,1;\Lambda}$}
The vanishing of $\Theta_{1,1;\Lambda_p}$ in Proposition \ref{propLp} is
due to the fact, that the lattice
$\Lambda_p$ has ``enough'' automorphisms. More precisely, we have the
following general result.

\begin{proposition}\label{big-auto}
Let $\Lambda \subset \edop^n$ be a lattice and suppose its orthogonal
group 
$\Orth_\Lambda:=\{\phi\in\Orth(n)\mid \phi(\Lambda)=\Lambda\}$ acts
irreducibly on $\edop^n$.  Then $\Theta_{1,1;\Lambda}$ is zero.
\end{proposition}
\begin{proof}
Let $G=\Orth_\Lambda$ be the orthogonal group of $\Lambda$.  We consider
the following function
\[ Q: \edop^ n \times \edop^n  \to \rdop \quad (x,y) \mapsto
\sum_{g \in G} \sca{g(x),y}^2 \, . \]

Since $\edop^n$ is an irreducible representation of $G$, a vector $y \ne
0$ cannot be orthogonal to all
vectors $\{ g(x) \}_{g \in G}$ unless $x=0$.  We conclude, that for fixed
$y \ne 0$ in $\edop^n$ the
quadratic form $Q_y: \edop^n \to \rdop$ which sends $x \mapsto Q(x,y)$ is
positive definite.  By definition
$Q_y$ is $G$-invariant, i.e. $Q_y(g(x))=Q_y(x)$.  The irreducibility of
the $G$-action on $\edop^n$ implies
that $Q_y$ is proportional to the euclidean quadratic form $x \mapsto
\|x\|^2$.  Thus, $Q_y(x)=\phi(y)
\cdot \|x\|^ 2$ for some real number $\phi(y)$ depending on $y$.

Analogously, we can define for a vector $x \in \edop^n$ the quadratic form
$Q_x(y)= Q(x,y)$. The group $G$
is a subgroup of $\Orth(n)$, hence $Q_x(gy)=Q_x(y)$. Eventually, we deduce
that $Q(x,y)=c \cdot
\|x\|^2\cdot \|y\|^2$ for some constant $c \in \rdop$.  We compute the
constant $c$ by averaging $x$ over
all vectors in the unit sphere $S^{n-1}$ in $\edop^n$, and taking
$y=\transp(1,0,\ldots,0)$.  On $S^{n-1}$
we consider the unique invariant measure $d\mu$ such that $\int_{S^{n-1}}
d\mu =1$, and so
\begin{align*}
c &= \int_{S^{n-1}} \sum_{g \in G} \sca{g(x),y}^2 d\mu(x) \\
& = \sum_{g \in G} \int_{S^{n-1}} \sca{g(x),y}^2 d\mu(x) &\text{ since $G$
acts isometrically we have}\\
&= \vert G\vert \int_{S^{n-1}} \sca{x,y}^2 d\mu(x) \\
&= \vert G\vert \int_{S^{n-1}} x_1^2 d\mu(x) &\text{ since
$\sum_{i=1}^nx_i^2=1$ we conclude that }\\
&= \frac{\vert G\vert}{n} \, .
\end{align*}
Using this expression we deduce from 
\[
\Theta_{1,1;\Lambda}(z)  =   \sum_{\lambda \in
\Lambda} \sum_{\mu \in \Lambda} \left( \sca{\lambda,\mu}^2-\frac{1}{n}
\|\lambda\|^2 \|\mu\|^2 \right) q^{\|\lambda\|^2+ \|\mu\|^2}\] 
the formula
\[ \Theta_{1,1;\Lambda}(z)  =  \sum_{[\lambda]  \in \Lambda/G} \quad 
\sum_{\mu \in \Lambda}
\frac{1}{\# {\rm Stab}(\lambda)} \sum_{g \in G} \left(
\sca{g(\lambda),\mu}^2-\frac{1}{n}
\|g(\lambda)\|^2 \|\mu\|^2 \right) q^{\|g(\lambda)\|^2+ \|\mu\|^2} 
\]
by considering the $G$-orbits of the element $\lambda \in \Lambda$ . The
sum over $G$ is already zero --
which follows immediately from $g\in G\subset \Orth(n)$ and $c=\vert
G\vert/n$ --, and therefore
$\Theta_{1,1;\Lambda}\equiv0$.
\end{proof}

\begin{remark}
In his diploma thesis T.~Alfs gives a formula for the
invariant $\Theta_{1,1;\Lambda}$ when
$\Lambda =\Lambda_1 \oplus \Lambda_2$ is a direct sum of two lattices.  He
shows in \cite[Satz 6.3]{Alf}
that we have
\[ \Theta_{1,1;\Lambda_1 \oplus \Lambda_2} =
\frac{n_1}{n_1+n_2} \Theta_{1,1;\Lambda_1} \Theta^2_{\Lambda_2}
+ \frac{n_2}{n_1+n_2} \Theta_{1,1;\Lambda_2} \Theta^2_{\Lambda_1}
+ \frac{2}{n_1n_2(n_1+n_2)^2}F_1( \Theta_{\Lambda_1}, \Theta_{\Lambda_2})^2
\]
where $\Lambda_i$ is a lattice in $\edop^{n_i}$ and
$F_1(\Theta_{\Lambda_1}, \Theta_{\Lambda_2})$ is
the first Rankin--Cohen differential operator see \cite[Section
1.3]{Zag}.
\end{remark}
\begin{remark}
As a consequence we obtain that for a lattice
$\Lambda$ with $\Theta_{1,1;\Lambda}=0$ we also have
$\Theta_{1,1;\Lambda^{\oplus m}}=0$ for all $m\geq 1$. 
Using Proposition \ref{big-auto} we deduce
the vanishing of $\Theta_{1,1}$ 
for all powers of root lattices $A_n^k$, $D_n^k$, $E_n^k$.
\end{remark}

\section{Lattices with degenerate differential}\label{sec:D0}
Let now $\Lambda \subset \edop^n$ be a lattice with corresponding
quadratic form $A$. Let $\mathcal{B}=\{ h_i \}
_{i=1,\ldots, m}$ with $m= \frac{n^2+n-2}{2}$ be an orthonormal basis of
the vector space of homogeneous
harmonic polynomials of degree two.  We set $\Theta_{h_i}(\Lambda):=
\sum_{\lambda \in \Lambda}
h(\lambda)q^{\|\lambda\|^2}$.

We have seen in Section \ref{split} that the vectors
$\Theta_{h_i}(\Lambda)$ in the vector space $\Mod$ are (up to
the scaling factor $2\pi i z$) the differentials of a basis of the tangent
space $T_A^0$ of $A$ in
$\Quad_n^+$. These functions are linearly dependent if and only if the
following determinant is zero:
\[ \det{}_{\mathcal B} \Diff \Theta_\Lambda(z):= \det \left(
\begin{array}{rrcr}
\Theta_{h_1} & \Theta_{h_2} & \cdots & \Theta_{h_m}\\\\
\frac{\dee}{2 \pi i \dee z} \Theta_{h_1} &\frac{\dee}{2 \pi i \dee z}
\Theta_{h_2} & \cdots & \frac{\dee}{2 \pi i \dee z} \Theta_{h_m}\\\\
\left( \frac{\dee}{2 \pi i \dee z} \right)^{2} \Theta_{h_1} &
\left( \frac{\dee}{2 \pi i \dee z} \right)^{2} \Theta_{h_2} & \cdots &
\left( \frac{\dee}{2 \pi i \dee z} \right)^{2} \Theta_{h_m}\\ \\
\vdots \qquad & \vdots \qquad & \ddots & \vdots \qquad\\\\
\left( \frac{\dee}{2 \pi i \dee z} \right)^{m-1} \Theta_{h_1} &
\left( \frac{\dee}{2 \pi i \dee z} \right)^{m-1} \Theta_{h_2} & \cdots &
\left( \frac{\dee}{2 \pi i \dee z} \right)^{m-1} \Theta_{h_m}\\ 
\end{array} \right)\; . \] 

The scaling factor $\frac{1}{2 \pi i}$ is included to make
differentiation closed in the ring
$\zdop[[q]]$. When $\Lambda$ is an integral lattice, then it turns out
that $\det{}_{\mathcal B}\Diff 
\Theta_\Lambda$ is a modular form. Suppose that $\Lambda$ is integral of
level $N$. Then the $\Theta_{h_i}$
are modular forms of weight $w=\frac{n}{2}+2$ for the group $\Gamma_0(N)$.
We write $\Theta_{\ul{h}}$ for
the row vector $\left( \Theta_{h_1}, \ldots,\Theta_{h_m} \right)$. We
deduce inductively from
\[ \Theta_{\ul{h}}(\gamma(z)) = (cz+d)^w\Theta_{\ul{h}}(z)
\quad \text{ for } \quad \gamma=\matzwei{a&b\\c&d} \in \Gamma_0(N) \, ,
\]
that
\[ \frac{\dee^k \Theta_{\ul{h}}}{\dee z^k} (\gamma(z))
= (cz+d)^{w+2k} \frac{\dee^k \Theta_{\ul{h}}}{\dee z^k}(z)
+ \sum_{l=0}^{k-1} a_{l,k} (cz+d)^{w+k+l} \frac{\dee^l
\Theta_{\ul{h}}}{\dee z^l}(z)
\]
with constants $a_{l,k}$ depending only on $l$, $k$, $w$ and $c$. 
Therefore, modulo the span of 
\[\left\{ (cz+d)^{w+k+l} \frac{\dee^l \Theta_{\ul{h}}}{\dee
z^l}\,\Big\vert\, l = 0,\ldots,k-1\right\}\,
\]
the vector $\frac{\dee^k \Theta_{\ul{h}}}{\dee z^k}$ behaves like a
modular form of weight $w+2k$.  When
computing the determinant above, we may hence assume that $\frac{\dee^k
\Theta_{\ul{h}}}{\dee z^k}$ is a modular form of weight  $w+2k$ since
adding of rows does not alter the
determinant. We conclude that $\det{}_{\mathcal B} \Diff\Theta_\Lambda$ is
a modular form of weight $m(w+m-1)=
\frac{(n+2)^2n(n-1)}{4}$. Since changing the orthonormal basis $\mathcal
B$ multiplies the original
determinant by $\pm1$, we define
$\det^2\Diff\Theta_\Lambda:=\left(\det{}_{\mathcal
B}\Diff\Theta_\Lambda \right)^2$ for
any orthonormal basis $\mathcal B$ of degree two harmonic polynomials and
obtain

\begin{proposition}\label{HLambda}
The function $\det^2 \Diff\Theta_\Lambda$ is a lattice invariant.  We have
an equivalence
\[ \det{}^2 \Diff\Theta_\Lambda = 0 \iff \text{ the differential of the
theta map is degenerated at }
\Lambda \, .\]
If $\Lambda$ is an integral lattice of level $N$, then
$\det^2\Diff\Theta_\Lambda$ is a modular form of
weight $\frac{(n+2)^2n(n-1)}{2}$ for $\Gamma_0(N)$.
\end{proposition}

We will write $\det\Diff\Theta_\Lambda$ for $\det{}_{\mathcal
B}\Diff\Theta_\Lambda$ whenever the basis
$\mathcal B$ is understood from the context or is irrelevant. For example when
considering whether $\det \Diff\Theta_\Lambda=0$.

\begin{remark}
Similar to the constructions in Section \ref{VanDiff} we can now construct
lattices $\Lambda$ such that
$\det \Diff\Theta_\Lambda=0$.
Geometrically
speaking, these are the lattices where the local Torelli theorem fails for
the theta map. Since all the
examples with vanishing differential from Section \ref{VanDiff} give also
lattices with degenerate
differential we give here only one example:\\ Let $\Lambda \subset \edop^n
$ be a lattice that is invariant
under a reflection $\rho: \edop^n \to \edop^n$. Then we have $\det
\Diff\Theta_\Lambda=0$. For a proof of
this assertion take two coordinates $x_1$ and $x_2$ on $\edop^n$ such that
$\rho^*x_i=(-1)^ix_i$. Then
$\Theta_{x_1x_2;\Lambda} = 0$, and so is $\det {\rm D}\Theta_\Lambda=0$.
\end{remark}

\begin{proposition}\label{GenLocTor}
{\bf (Generic local Torelli theorem)} For a general lattice $\Lambda
\subset \edop^n$ we have $\det\Diff 
\Theta_\Lambda\ne 0$. This means, the $\Theta$-map is generically locally
injective.
\end{proposition}
\begin{proof}

It is enough to show that there exists a lattice $\Lambda$ such that the
differential of the $\Theta$-map
is of full rank in $\Lambda$. This will imply the statement of the
proposition.

We start with a positive quadratic form $A\in \Quad_n^+$ given by a
symmetric matrix
$A=(a_{ij})_{i,j=1,\ldots n}$. We assume furthermore that the set $\{
a_{ij}\}_{1 \leq i \leq j \leq n}$ is
linear independent in the $\qdop$ vector space $\rdop$. This linear
independence implies by definition that
for two vectors $\ul{m}$ and $\ul m'$ in $\zdop^n$ we can conclude from
$A(\ul m) = A(\ul m')$ that $\ul m
=\pm\ul m'$.

Let now $\Lambda \subset \edop^n$ be the lattice associated to $A$.  We
claim that $\det\Diff
\Theta_{\Lambda}\ne 0$. On the contrary, assume that $\det
\Diff\Theta_{\Lambda}=0$. This implies by
Proposition \ref{HLambda} that the functions $\{ \Theta_{h_i;\Lambda} \}$
are linearly dependent where the
$h_i$ form a basis of the harmonic homogeneous quadratic forms on
$\edop^n$. Therefore there exists a
harmonic quadratic form $h$ such that $\Theta_{h;\Lambda} =0$.  Since for
two $\lambda,\lambda' \in
\Lambda$ we have $\|\lambda \|^2=\|\lambda'\|^2$ only for $\lambda'=\pm
\lambda$, we conclude that the
coefficient of $\exp(2 \pi i \|\lambda\|^2 z)$ in $\Theta_{h;\Lambda}$ is
$h(\lambda)+h(-\lambda)=2h(\lambda)$.  This implies $h\vert_{\Lambda}
\equiv 0$, and thus $h=0$. Hence
D$\Theta_{\Lambda}$ has full rank.
\end{proof}

\section{Example: Lattices in dimension two}
\subsection{Lattices with vanishing differential in dimension 2}
\begin{proposition}
The above two examples -- the Gaussian and the Eisenstein integers -- are
up to scalar multiples the only
examples of lattices with vanishing $\Theta$-differential in dimension two.
\end{proposition}
\begin{proof}
We consider the lattice vectors $\lambda \ne 0$ of minimal length in our
lattice $\Lambda \subset \edop^2$.
These appear pairwise, since $-\lambda $ and $\lambda $ have the same
length. Furthermore, we scale
$\Lambda$ such that $\| \pm \lambda\|=1$ holds for a pair of vectors of
minimal positive length. There are
three cases:\\
{\em Case 1: There exists one pair $(\lambda,-\lambda)$ of vectors of
minimal length.}
The coefficient $c_2$ of $q^2$ in  $\Theta_{1,1;\Lambda}$ is given by \[
c_2= 4(2
\sca{\lambda,\lambda}^2-\|\lambda\|^4) =4 \, .\] Thus,
$\Theta_{1,1;\Lambda}$ can not be zero.

{\em Case 2: There exists two pairs $(\lambda,-\lambda)$ $(\mu,-\mu)$ of
vectors of minimal length.} Again
we compute the coefficient $c_2$ of $q^2$ to be
\[ c_2= 4(2 \sca{\lambda,\lambda}^2-\|\lambda\|^4) + 
4(2 \sca{\mu,\mu}^2-\|\mu\|^4) + 8(2
\sca{\lambda,\mu}^2-\|\lambda\|^2\|\mu\|^2)  = 16 \sca{\lambda,\mu}^2 \,
.  \]
Thus, we see that for $\Theta_{1,1;\Lambda}=0$ we must have
$\sca{\lambda,\mu}=0$. This is the case of the
Gaussian integers (Example 1 in \S \ref{examp}).

{\em Case 3: There exists three pairs of vectors of minimal length.}
Indeed this happens only for the
Eisenstein lattice (Example 2 in \S \ref{examp}).
\end{proof}

\subsection{Lattices with degenerate differential in dimension two}
We will now investigate for which lattices in dimension two the
differential is degenerate. We consider the
harmonic functions $h_1(x,y)=x^2-y^2$ and $h_2(x,y)=xy$. They form a
orthogonal basis of the harmonic
functions of degree two in two variables. Thus, for a lattice $\Lambda
\subset \edop^2$ the differential of
the $\Theta$-map (restricted to $T_A^0$) is spanned by $\Theta_{h_1}$ and
$\Theta_{h_2}$ with
\[ \Theta_{h_i} = \sum_{\lambda \in \Lambda}
h_i(\lambda)q^{\|\lambda\|^2} \, .\]
The differential degenerates whenever the difference
$\Theta_{h_1}\Theta_{h_2}'-\Theta_{h_1}'\Theta_{h_2}$
is zero. We compute hence this function, which is the function $\det
\Diff\Theta_\Lambda$ with the
appropriate scaling (cf. Section \ref{sec:D0}):
\[\det \Diff\Theta_\Lambda=\frac{1}{2 \pi i}
(\Theta_{h_1}\Theta_{h_2}'-\Theta_{h_1}'\Theta_{h_2}) =
\sum_{(\lambda,\mu) \in \Lambda^2}
h_1(\lambda)h_2(\mu)(\|\mu\|^2-\|\lambda\|^2)
q^{\|\lambda\|^2+\|\mu\|^2} . \]
We can symmetrize this expression by interchanging the lattice elements
$\lambda$ and $\mu$ in this
summation to obtain:
\[2\det \Diff\Theta_\Lambda =  \sum_{(\lambda,\mu) \in \Lambda^2}
(h_1(\lambda)h_2(\mu)-h_1(\mu)h_2(\lambda)) (\|\mu\|^2-\|\lambda\|^2)
q^{\|\lambda\|^2+\|\mu\|^2} \, .\]
Expanding these terms yields
\begin{eqnarray}\label{eq5}
\det \Diff\Theta_\Lambda&=& \frac{1}{2} \sum_{(\lambda,\mu) \in \Lambda^2}
\sca{\lambda,\mu} \det(\lambda,\mu) (\|\mu\|^2-\|\lambda\|^2)
q^{\|\lambda\|^2+\|\mu\|^2} \,,
\end{eqnarray}
where for $\lambda=(\lambda_1,\lambda_2)$, and $\mu=(\mu_1,\mu_2)$ we write
$\det(\lambda,\mu):=\lambda_1\mu_2-\lambda_2\mu_1$ for the usual
determinant.  Whereas $\sca{\lambda,\mu}$,
$\|\lambda\|^2$, and $\|\mu\|^2$ are $\Orth(2)$-invariants the determinant
transforms like
$\det(\gamma(\lambda),\gamma(\mu)) = \det(\gamma)\det(\lambda,\mu)$.  Thus
$\det \Diff\Theta_\Lambda$ is
a half-invariant for the orthogonal group, since $\det
\Diff\Theta_{\gamma(\Lambda)} = \det(\gamma) \det
\Diff\Theta_\Lambda$.  By Proposition \ref{HLambda} $\det
\Diff\Theta_\Lambda$ is zero if and only if
the differential of the $\Theta$ map is degenerate. Next we determine its
degeneration locus.

\begin{lemma}\label{deg-dim-2}
For a lattice $\Lambda \subset \edop^2$ we have $\det
\Diff\Theta_\Lambda=0$ in the following three
cases:
\begin{enumerate}
\item $\Lambda$ is spanned by two orthogonal vectors.
\item $\Lambda$ is spanned by a pair $(\lambda_1,\lambda_2)$ of vectors
with $2\sca{\lambda_1,\lambda_2}=\|\lambda_1\|^2$.
\item $\Lambda$ is spanned by two vectors of the same length.
\end{enumerate}
\end{lemma}
\begin{proof}
In all three cases we use the following fact: If there exists an element
$\gamma \in \Orth(2)$ of
determinant $\det(\gamma)=-1$ such that $\gamma(\Lambda)=\Lambda$, then
the formula $\det\Diff 
\Theta_{\gamma(\Lambda)}=\det(\gamma)\det \Diff\Theta_\Lambda$ implies
$\det \Diff\Theta_\Lambda =
0$.  Identifying $\edop^2$ with the complex numbers $\cdop$ the maps
$\gamma$ are reflections $r_\tau$ with
$r_\tau(z)= \frac{\tau \ol z}{\ol \tau}$ is the reflection on the real
line generated by $\tau$.\\
\begin{enumerate}
\item If $\Lambda= \zdop \lambda_1 \oplus \zdop \lambda_2$ with $\lambda_1
\bot \lambda_2$, then we can
take the reflection $r_{\lambda_1}$ or $r_{\lambda_2}$.
\item If $\Lambda= \zdop \lambda_1 \oplus \zdop \lambda_2$ with
$2\sca{\lambda_1,\lambda_2}=\|\lambda_1\|^2$, then the reflection
$r_{\lambda_1}$ preserves $\Lambda$.
\item If  $\Lambda= \zdop \lambda_1 \oplus \zdop \lambda_2$ with
$\|\lambda_1\|= \|\lambda_2\|$, then the
reflection $r_{\lambda_1+\lambda_2}$ interchanges $\lambda_1$ and
$\lambda_2$.
\end{enumerate}
Thus in all the above cases we see that $\det \Diff\Theta_\Lambda=0$. If
$\Lambda$ is a lattice which
does not belong to the three cases above, then we see that for the pairs
$0$, $(\pm \lambda_1)$, $(\pm
\lambda_2)$, $(\pm \lambda_3), \ldots $ of lattice vectors ordered by
length we have strict inequalities
$0< \|l_1 \| < \|l_2\|<\|l_3\|$. Using the $q$-expansion of $\det
\Diff\Theta_\Lambda$ given in
(\ref{eq5}) we can compute the coefficient of
$q^{\|\lambda_1\|^2+\|\lambda_2\|^2}$ to be
$2\sca{\lambda_1,\lambda_2} \det(\lambda_1,\lambda_2) \left(
\|\lambda_2\|^2-\|\lambda_1\|^2 \right)$. In
this remaining case, each of these three factors (besides $2$) are non
zero, giving $\det {\rm
D}\Theta_\Lambda\neq0$.
\end{proof}

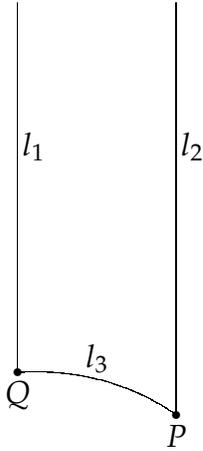
\setlength{\unitlength}{0.4pt}
\begin{picture}(250,300)(150,200)
\linethickness{0.1pt}
\bezier{300}(200,100)(285,105)(350,60)
\put(350,60){\line(0,1){390}}
\put(200,100){\line(0,1){350}}
\put(350,60){\circle*{7}}
\put(200,100){\circle*{7}}
\put(350,30){\makebox(0,0)[b]{$P$}}
\put(200,65){\makebox(0,0)[b]{$Q$}}
\put(215,300){\makebox(0,0)[b]{$l_1$}}
\put(365,300){\makebox(0,0)[b]{$l_2$}}
\put(275,100){\makebox(0,0)[b]{$l_3$}}
\end{picture} 
\hspace{2em}
\begin{minipage}{9cm}
\[ F= \left\{ \tau \in \hdop \, \text{ with }  \|\tau \| \geq 1 \text{,
and }
0 \leq \Re(\tau) \leq \frac{1}{2} \right\}\]
The differential of the $\Theta$ map is degenerate at $\tau
\in F \iff \tau \in (l_1 \cup l_2 \cup l_3)$.\\

The lattices parameterized by $l_i$ correspond to the case (i) in Lemma
\ref{deg-dim-2}. For example, the
lattices parameterized by $\tau \in l_1$ are spanned by two orthogonal
vectors.\\

The differential of the $\Theta$ map maximally degenerates at the two
points $P$ and $Q$.


\text{ }

\end{minipage}\\
Figure: The fundamental domain $F$ for lattices in the upper half plane
$\hdop$.

\subsection{The weak local Torelli theorem}
The vanishing of $\det \Diff\Theta_\Lambda$ along the curves $l_i$ may at
a first glance be interpreted
as a failure of the local Torelli theorem.  However, any continuous
lattice invariant on the upper half
plane $\hdop$ with values in a vector space $V$ must send the two lattices
generated by $(1,a\sqrt{-1} \pm
b)$ for real $a$ and $b$ to the same element of $V$, independent of the
sign. Thus, the differential of the
invariant must be degenerate at the point $(1,a\sqrt{-1})$. Hence the
differential is degenerate along
$l_1$. Analogously we see, that the differential must be degenerate along
$l_2$ and $l_3$.

The global Torelli theorem is an elementary exercise for rank two
lattices.  Indeed, the first three
lengths of lattice vectors and their multiplicities give the Gram matrix
of the lattice.

\subsection*{Acknowledgment}
This work has been supported by the SFB/TR 45
``Periods, moduli spaces and arithmetic of algebraic varieties''.

\end{document}